\documentclass[reqno]{amsart}

\newtheorem{thm}{Theorem}[section]
\newtheorem{lem}[thm]{Lemma}

\newtheorem{prop}[thm]{Proposition}

\newtheorem{rem}[thm]{Remark}

\numberwithin{equation}{section}

\newcommand{\R}{\mathbb{R}}
\newcommand{\C}{\mathbb{C}}

\newcommand{\N}{\mathbb{N}}
\newcommand{\A}{\mathbb{A}}

\newcommand{\E}{\mathbb{E}}

\newcommand{\ml}{\mathcal{L}}

\newcommand{\rd}{\mathrm{d}}

\newcommand{\bqn}{\begin{equation}}
\newcommand{\eqn}{\end{equation}}
\newcommand{\bqnn}{\begin{equation*}}
\newcommand{\eqnn}{\end{equation*}}
\newcommand{\bear}{\begin{eqnarray}} 
\newcommand{\eear}{\end{eqnarray}} 
\newcommand{\bean}{\begin{eqnarray*}} 
\newcommand{\eean}{\end{eqnarray*}} 
\newcommand{\bs}{\begin{split}}
\newcommand{\es}{\end{split}}

\newcommand{\dhr}{\mathrel{\lhook\joinrel\relbar\kern-.8ex\joinrel\lhook\joinrel\rightarrow}}

\begin{document}

\title[Age-Structured Diffusive Populations]{Some Remarks About the Semigroup Associated to Age-Structured Diffusive Populations}

\author{Christoph Walker}
\email{walker@ifam.uni-hannover.de}
\address{Leibniz Universit\"at Hannover\\ Institut f\" ur Angewandte Mathematik \\ Welfengarten 1 \\ D--30167 Hannover\\ Germany}

\begin{abstract}
We consider linear age-structured population equations with diffusion. 
Supposing maximal regularity of the diffusion operator, we characterize the generator and its spectral properties of the associated strongly continuous semigroup. In particular, we provide conditions for stability of the zero solution and for asynchronous exponential growth.
\end{abstract}

\keywords{Age structure, maximal regularity, semigroups of linear operators, asynchronous exponential growth.}
\subjclass[2010]{47D06, 92D25, 47A10, 34G10}

\maketitle

\section{Introduction}

This work is dedicated to age-structured diffusive population dynamics governed by the abstract linear equations
\begin{align}
&\partial_t u\,+\, \partial_au \, +\,     A(a)\,u\, =0\ , && t>0\ ,\quad a\in (0,a_m)\ ,\label{1}\\ 
&u(t,0)\, =\, \int_0^{a_m}b(a)\, u(t,a)\, \rd a\ ,& & t>0\ ,  \label{2}\\
&u(0,a)\, =\,  \phi(a)\ , & &\hphantom{t>0\ ,\quad} a\in (0,a_m)\ .\label{3}
\end{align}
Here, $u=u(t,a)$ is a function taking positive values in some ordered Banach space $E_0$. In applications it represents the density at time $t$ of a population of individuals structured by age $a\in J:=[0,a_m)$, where $a_m\in (0,\infty]$ is the maximal age. Note that the age interval $J$ may be unbounded. For fixed $a\in J$, the operator 
\bqn\label{AA}
A(a)=A_0(a)+\mu(a)
\eqn
involves spatial movement of individuals described by $A_0(a)$  and death processes of individuals with mortality rate $\mu(a)$ and is assumed to be an (unbounded) linear operator $A(a):E_1\subset E_0\rightarrow E_0$. The nonlocal age-boundary condition \eqref{2} represents birth processes with birth rate $b$ while $\phi$ in \eqref{3} describes the initial population.

Equations \eqref{1}-\eqref{3} and variants thereof, e.g. for constant or time-dependent operators $A$, have been investigated by many authors, for example see \cite{GuoChan,Huyer,Rhandi,RhandiSchnaubelt_DCDS99,ThiemeDCDS,WalkerDIE07,WebbSpringer} and the references therein though this list is far from being complete.

Recall, e.g. from \cite{WebbSpringer}, that a strongly continuous semigroup in $L_1(J,E_0)$ can be associated with \eqref{1}-\eqref{3} if $A$ is independent of age and generates itself a strongly continuous semigroup on $E_0$, see also \cite{GuoChan,WalkerDIE07}. This is derived upon formally integrating \eqref{1} along characteristics giving the semigroup rather explicitly. The approach has been extended to investigate the well-posedness of models featuring nonlinearities in the operator $A=A(t,u)$ or in the birth rates $b=b(t,u)$ \cite{WalkerEJAM08,WalkerDCDSA10}.

A slightly different approach has been chosen in \cite{ThiemeDCDS}. On employing methods for positive perturbations of semigroups it has been shown that a strongly continuous semigroup for \eqref{1}-\eqref{3} in $L_1(J,E_0)$ is obtained as the derivative of an integrated semigroup. Moreover, this strongly continuous semigroup is shown to enjoy certain compactness properties and to exhibit asynchronous exponential growth, i.e. it stabilizes as
$t\rightarrow\infty$ to a one-dimensional image of the state space of initial values, after
multiplication by an exponential factor in time. This result has been recovered as a particular case in  \cite{RhandiSchnaubelt_DCDS99} (also see \cite{Rhandi}), where time-dependent birth rates have been included by means of perturbation techniques of Miyadera type. It is noteworthy that the general results of \cite{ThiemeDCDS} apply as well to other situations than $A_0$ describing spatial diffusion.

The strongly continuous semigroup from \cite{WebbSpringer,WalkerDIE07} associated  to \eqref{1}-\eqref{3} has the advantage that certain properties --- like regularizing effects in the case that $-A$ generates an analytic semigroup, being of utmost importance in nonlinear equations, see \cite{WalkerEJAM08,WalkerDCDSA10} --- can be read off its formula rather easily, see Theorem~\ref{Prop1} below and the subsequent remarks. The domain of the generator of this semigroup is in general not fully identified, cf. \cite{WalkerDIE07,WalkerEJAM08}.
The objective of the present paper is to characterize the (domain of the) infinitesimal generator of the strongly continuous semigroup associated  to \eqref{1}-\eqref{3} and  to investigate its spectral properties in the case that the operator $A$ has the property of maximal $L_p$-regularity. This assumption is satisfied in many applications, e.g. when $A_0$ in \eqref{AA} is a second order elliptic differential operator in divergence form. Maximal regularity provides an adequate functional analytic setting for the characterization of the generator of the semigroup associated  to \eqref{1}-\eqref{3} in the phase space $L_p(J,E_0)$ and its resolvent, see Theorem~\ref{C1} below. Knowing the generator precisely, we shall then investigate its growth bound and derive a stability result for the trivial solution, see Theorem~\ref{T1} below. We also provide in Theorem~\ref{C14} a condition that implies asynchronous exponential growth of the semigroup. 

Besides a precise description of the semigroup with its generator and stability of the trivial solution, we thus obtain a similar result on asynchronous exponential growth as in~\cite{ThiemeDCDS, RhandiSchnaubelt_DCDS99} by another approach being  inspired by the results in \cite{WebbBook} that were dedicated to the non-diffusive scalar case.  We shall point out, however, that the results and the approach presented herein shall serve as a basis for a future investigation of qualitative aspects of solutions to models featuring nonlinearities in the diffusion part and in the age-boundary condition, i.e. for diffusion operators of the form $A=A(u)$ and birth rates $b=b(u)$, by means of linearization and perturbation techniques.
Finally, from a technical point of view it seems to be worthwhile to point out that the cases of a finite or infinite maximal age $a_m$ is treated simultaneously herein.

\section{The Semigroup and its Generator}

\subsection{Notation and Assumptions}

Given a closed linear operator $\mathcal{A}$ on a Banach space, we let $\sigma(\mathcal{A})$ and $\sigma_p(\mathcal{A})$ denote its spectrum and point spectrum, respectively. The essential spectrum
 $\sigma_e(\mathcal{A})$ of $\mathcal{A}$ consists of those spectral points $\lambda$ of $\mathcal{A}$ such that the image $\mathrm{im}(\lambda-\mathcal{A})$ is not closed, or $\lambda$ is a limit point of $\sigma(\mathcal{A})$, or the dimension of the kernel $\mathrm{ker}(\lambda-\mathcal{A})$ is infinite. The peripheral spectrum $\sigma_0(\mathcal{A})$ is defined as
$
\sigma_0(\mathcal{A}):=\{\lambda\in\sigma(\mathcal{A})\,;\, \mathrm{Re}\,\lambda=s(\mathcal{A})\}
$, 
where $s(\mathcal{A}):=\sup\{\mathrm{Re}\,\lambda\,;\,\lambda\in \sigma(\mathcal{A})\}$ denotes the spectral bound of $\mathcal{A}$.
The resolvent set $\C\setminus \sigma(\mathcal{A})$ is denoted by $\varrho(\mathcal{A})$.

Throughout $E_0$ is a real Banach lattice ordered by a closed convex cone $E_0^+$. However, we do not distinguish $E_0$ from its complexification in our notation as no confusion seem likely. Recall that a $u\in E_0^+$ is quasi-interior if $\langle f, u\rangle >0$ for all $f$ in the dual space $E_0'$ with $f>0$ .

Let $E_1$ be a densely and compactly embedded subspace of $E_0$. We fix $p\in (1,\infty)$, put \mbox{$\varsigma:=\varsigma(p):=1-1/p$} and set 
$$
E_\varsigma:=(E_0,E_1)_{\varsigma,p}\ ,\quad E_\theta:= (E_0,E_1)_\theta 
$$
for $\theta\in [0,1]\setminus\{1-1/p\}$ with $(\cdot,\cdot)_{\varsigma,p}$ being the real interpolation functor and $(\cdot,\cdot)_\theta$ being any admissible interpolation functor. We equip these interpolation spaces with the order naturally induced by $E_0^+$. Observe that $E_\theta$ embeds compactly in $E_\vartheta$ provided $0\le\vartheta<\theta\le 1$. We put
$$
\E_0:=L_p(J,E_0)\ ,\qquad \E_1:=L_p(J,E_1)\cap W_p^1(J,E_0)
$$ and recall that
\bqn\label{6}
\E_1\hookrightarrow BUC(J,E_{\varsigma})
\eqn
according to, e.g. \cite[III.Thm.4.10.2]{LQPP}, where $BUC$ stands for the bounded and uniformly continuous functions. In particular, the trace $\gamma_0 u:= u(0)$ is well-defined for $u\in \E_1$ and $\gamma_0\in\ml(\E_1,E_\varsigma)$. Let $\E_0^+$ denote the functions in $\E_0$ taking almost everywhere values in $E_0^+$. Note that $\E_0$ is a Banach lattice. We further assume that
\bqn\label{10a}
A\in L_\infty(J,\mathcal{L}(E_1,E_0))\ ,\quad
\sigma+A\in C^\rho(J,\mathcal{H}(E_1,E_0;\kappa,\nu))
\eqn
for some $\rho,\nu>0$, $\kappa\ge 1$, $\sigma\in\R$. Here $\mathcal{H}(E_1,E_0;\kappa,\nu)$ consists of all negative generators $-\mathcal{A}$ of analytic semigroups on $E_0$ with domain $E_1$ such that $\nu+\mathcal{A}$ is an isomorphism from $E_1$ to $E_0$ and
$$
\kappa^{-1}\le\frac{\|(\lambda+\mathcal{A})x\|_{E_0}}{\vert\lambda\vert \,\| x\|_{E_0}+\|x\|_{E_1} }\le\kappa\ ,\quad x\in E_1\setminus\{0\}\ ,\quad \mathrm{Re}\, \lambda\ge \nu\ .
$$
Note that $A$ generates a parabolic evolution operator $\Pi(a,\sigma)$, $0\le\sigma\le a<a_m$, on $E_0$ with regularity subspace $E_1$ according to \cite[II.Cor.4.4.2]{LQPP} and there are $M\ge 1$ and $\varpi\in\R$ such that
\bqn\label{10aa}
\|\Pi(a,\sigma)\|_{\ml (E_\alpha)}+(a-\sigma)^{\alpha-\beta_1}\|\Pi(a,\sigma)\|_{\ml (E_\beta,E_\alpha)}\le Me^{-\varpi (a-\sigma)}\ ,\quad 0\le \sigma\le a< a_m\ ,
\eqn
for $0\le\beta_1\le\beta<\alpha\le 1$ with $\beta_1<\beta$ if $\beta>0$, see \cite[II.Lem.5.1.3]{LQPP}.
We further assume that $\Pi(a,\sigma)$ is positive for $0\le \sigma\le a< a_m$  and that 
\bqn\label{10}
\varpi>0\quad\text{if}\quad a_m=\infty\ .
\eqn
Moreover, we assume that
\bqn\label{12}
\begin{aligned}   
&\text{for each $\mathrm{Re}\ \lambda >-\varpi$, the operator}\ A_\lambda:=\lambda+A\  \text{has maximal}\ L_p\text{-regularity,}\\ 
&\text{that is,}\ (\partial_a+A_\lambda,\gamma_0):\E_1\rightarrow\E_0\times E_{\varsigma}\ \text{is an isomorphism}\ .
\end{aligned}
\eqn 
Let the birth rate $b$ be such that
\bqn\label{13}
b\in L_{\infty}(J,\ml(E_\theta))\cap L_{p'}(J,\ml(E_\theta))\ ,\qquad b(a)\in \ml_+(E_0)\ ,\quad a\in J\ ,
\eqn
for $\theta\in [0,1]$, where $p'$ is the dual exponent of $p$. We also assume that
\bqn\label{14}
b(a)\Pi(a,0)\in \ml_+(E_0) \ \text{is irreducible for $a$ in a subset of $J$ of positive measure}\ .
\eqn
Some of the assumptions above are redundant if $a_m<\infty$. For instance, if $a_m<\infty$ and $$A\in C^\rho([0,a_m],\ml(E_1,E_0))$$ is such that $-A(a)$ generates an analytic semigroup on $E_0$ for each $a\in J$, then \eqref{10a} holds. We shall furthermore point out that not all assumptions will be needed in this strength but are imposed for the sake of simplicity.  In particular, if $\mu$ being a real-valued nonnegative and locally integrable function and $A(a)=A_0(a)+\mu(a)$ as in the introduction, then it suffices that $A_0$ satisfies \eqref{10a} for what follows by keeping in mind that $$\Pi(a,\sigma)=e^{-\int_\sigma^a\mu(r)\rd r} U(a,\sigma)$$ with $U$ denoting the evolution operator associated with $A_0$. Also, \eqref{13} is not required for the whole range of $\theta\in [0,1]$.

We remark that the assumptions above are satisfied in many applications with $A$ describing spatial diffusion, for example see \cite[Sect.3]{WalkerSIMA09}, \cite[Sect.3]{WalkerAMPA11}. For details about parabolic evolution operators and operators having maximal regularity we refer the reader, e.g., to \cite{LQPP}. A summary on positive operators in ordered Banach spaces can be found e.g. in \cite{DanersKochMedina}.
\\

Due to \eqref{12}, the operator $A_\lambda=\lambda+A$ has maximal $L_p$-regularity on $J$ for each $\lambda\in\C$ provided $\mathrm{Re}\, \lambda>-\varpi$ if $a_m=\infty$ and it generates a parabolic evolution operator 
$$
\Pi_\lambda(a,s):=e^{-\lambda (a-\sigma)}\,\Pi(a,\sigma)\ , \quad 0\le s\le a< a_m\ ,
$$
on $E_0$. Consequently, the unique 
solution $\phi\in\E_1$ to
$$
\partial_a\phi+A_\lambda(a)\phi=f(a)\ ,\quad a\in (0,a_m)\ ,\qquad \phi(0)=\phi_0
$$
for $\phi_0\in E_\varsigma$ and $f\in\E_0$
is given by
$$
\phi(a)=\Pi_\lambda(a,0)\phi_0+\int_0^a\Pi_\lambda(a,s)\, f(s)\, \rd s\ ,\quad a\in J\ .
$$
In particular, $\Pi_\lambda(\cdot,0)\in\ml(E_\varsigma,\E_1)$ for  $\lambda\in\C$ with $\mathrm{Re}\, \lambda>-\varpi$ if $a_m=\infty$.

\subsection{The Semigroup and its Generator}

On integrating \eqref{1} along characteristics we formally derive that the solution $[S(t)\phi](a):=u(t,a)$ to \eqref{1}-\eqref{3} is given by
  \bqn\label{100}
     \big[S(t) \phi\big](a)\, :=\, \left\{ \begin{aligned}
    &\Pi(a,a-t)\, \phi(a-t)\ ,& & 0\le t\le a<a_m\ ,\\
    & \Pi(a,0)\, B_\phi(t-a)\ ,& & 0\le a<a_m\, ,\, t>a\ ,
    \end{aligned}
   \right.
    \eqn
with $B_\phi:=u(\cdot,0)$ satisfying according to \eqref{2} the Volterra equation 
    \bqn\label{5}
    B_\phi(t)\, =\, \int_0^t h(a)\, b(a)\, \Pi(a,0)\, B_\phi(t-a)\ \rd
    a\, +\, \int_0^{a_m-t} h(a)\, b(a+t)\, \Pi(a+t,a)\, \phi(a)\ \rd a\ ,\quad
    t\ge 0\ ,
    \eqn
with cut-off function $h(a):=1$ if $a\in(0,a_m)$ and $h(a):=0$ otherwise. Note that
  \bqn\label{6a}
    B_\phi(t)\, =\, \int_0^{a_m}  b(a)\, \big[S(t)\phi\big](a)\ \rd a\ ,\quad t\ge 0\ .
    \eqn
To make the formal integration rigorous, we first observe:

\begin{lem}\label{A}
There exists a mapping $[\phi\mapsto B_\phi]\in\ml \big(\E_0,
C(\R^+,E_0)\big)$ such that $B_\phi$ is the unique solution to~\eqref{5}. If
 $\phi\in \E_0^+$, then $B_\phi(t)\in E_0^+$ for $t\ge 0$. 
Given $\theta\in [0,1]$, there is $N:=N(\theta)>0$ such that
    \bqn\label{66}
    \| B_\phi (t)\|_{E_\theta}\, \le\, N\, t^{-\theta}
    e^{(-\varpi +\zeta(\theta)) t}\,\| \phi\|_{\E_0}\ ,\quad t>0 \
    ,
    \eqn
where $\zeta(\theta):=(1+\theta) M \|b\|_{L_\infty(J,\ml(E_\theta))}$.
\end{lem}

\begin{proof}
The proof is straightforward by standard arguments,  similar statements are found in  \cite[Thm.4]{WebbSpringer} and \cite[Lem.2.1]{WalkerDIE07}. We only note that  one obtains, for $t>0$, on applying \eqref{10aa} to \eqref{5} and on using \eqref{13},
\bqnn
e^{\varpi t}\|B_\phi(t)\|_{E_\theta}\le M \|b\|_{L_\infty(J,\ml(E_\theta))}\int_0^t e^{\varpi a}\, \| B_\phi(a)\|_{E_\theta}\ \rd a + M  \|b\|_{L_{p'}(J,\ml(E_\theta))}\|\phi\|_{\E_0} t^{-\theta}
\eqnn
and thus \eqref{66} follows from the singular Gronwall's inequality \cite[II.Cor.3.3.2]{LQPP}.
\end{proof}

Along the lines of \cite[Thm.4]{WebbSpringer} (for the case $p=1$) and on using \eqref{10aa} and \eqref{66} (also see \cite{WalkerDIE07}) one easily proves the following:

\begin{thm}\label{Prop1}
$\{S(t)\,;\,  t\ge 0 \}$ given in \eqref{100} is a strongly continuous positive semigroup in~$\E_0$ with $$ \sup_{t\ge 0}\, e^{t(\varpi-\zeta)}\|S(t)\|_{\ml(\E_0)} <\infty\ ,$$
where $\zeta:=\zeta(0)=M \|b\|_{L_\infty(J,\ml(E_0))}$.
\end{thm}

Though we shall not use it in the following let us note that the semigroup $\{S(t)\,;\,  t\ge 0 \}$ inherits  regularizing properties from the parabolic evolution operator stated in \eqref{10aa} , e.g. there holds
$$ 
\|S(t)\phi\|_{L_p(J,E_\theta)} \le c(\theta) \, t^{-\theta}\, e^{t(-\varpi+\zeta(\theta))}\, \|\phi\|_{\E_0}\ , \quad t>0\ ,\quad\phi\in\E_0\ ,\quad \theta\in [0,1/p)\ .
$$

Let  $-\A$ denote the generator of the semigroup $\{S(t)\,;\,  t\ge 0 \}$.
Based on the assumption of maximal regularity of the operator $A$ in \eqref{1}, we now fully characterize $-\A$. First, recall that $\{\lambda\in\C\,;\,\mathrm{Re}\, \lambda >\omega(-\A)\}$ is a subset of the resolvent set $\varrho(-\A)$, where the growth bound $\omega(-\A)$ is given by
$$
\omega(-\A):=\lim_{t\rightarrow\infty}\frac{1}{t} \log \|S(t)\|\ .
$$ 
Note that
Theorem~\ref{Prop1} entails
\bqn\label{23}
\omega(-\A)\le -\varpi+\zeta\ .
\eqn
Let $\lambda\in\C$ be such that $\mathrm{Re}\, \lambda >-\varpi$ if $a_m=\infty$. Observe that the solution to
$$
\partial_a\phi+A_\lambda(a)\phi=0\ ,\quad a\in (0,a_m)\ ,\qquad \phi(0)=\int_0^{a_m}b(a)\, \phi(a)\ \rd a\ ,
$$
is given by
$$
\phi(a)=\Pi_\lambda(a,0)\phi(0)\ ,\quad a\in (0,a_m)\ ,\qquad \phi(0)=Q_\lambda\phi(0)\ ,
$$
where $$
Q_\lambda:=\int_0^{a_m}b(a)\,  \Pi_\lambda(a,0)\ \rd a\ .
$$
A we shall see, the spectrum of $-\A$ and thus the asymptotic behavior of solutions to \eqref{1}-\eqref{3} is determined by the spectral radii of the $\lambda$-dependent family $Q_\lambda$.
From \eqref{10aa}, \eqref{10}, and \eqref{13} we deduce  the regularizing property 
\bqn\label{rp}
Q_\lambda\in\mathcal{L}(E_0,E_\theta)\cap \mathcal{L}(E_{1-\theta},E_1)\ ,\qquad\theta\in [0,1)\ ,
\eqn
and hence \mbox{$Q_\lambda\vert_{E_\theta}\in\mathcal{L}(E_\theta)$} is compact for $\theta\in [0,1)$ due to the compact embedding of $E_\alpha$ in $E_\beta$ for $0\le\beta<\alpha<1$. Consequently, \mbox{$\sigma(Q_\lambda\vert_{E_\theta})\setminus\{0\}$} consists only of eigenvalues.

\begin{lem}\label{Q}
Let $\lambda\in\R$ with $\lambda >-\varpi$ if $a_m=\infty$. Then the spectral radius $r(Q_\lambda)$ is positive and a simple eigenvalue of $Q_\lambda\in\ml(E_0)$ with an  eigenvector in $E_1$ that is quasi-interior in $E_0^+$. It is the only eigenvalue of $Q_\lambda$ with a positive eigenvector. Moreover, $\sigma(Q_\lambda\vert_{E_\theta})\setminus\{0\}=\sigma(Q_\lambda)\setminus\{0\}$ for $\theta\in [0,1)$. 
\end{lem}

\begin{proof}
Since $Q_\lambda\in\mathcal{L}(E_0)$ is compact and irreducible according to \eqref{14} (see the proof of \cite[Lem.2.1]{WalkerAMPA11}), it is a classical result that the spectral radius $r(Q_\lambda)$ is positive and a simple eigenvalue of $Q_\lambda$ with a quasi-interior eigenvector \cite[Thm.12.3]{DanersKochMedina}. This eigenvector belongs 
to $E_1$ owing to \eqref{rp}. The regularizing property \eqref{10aa} also ensures the last statement.
\end{proof}

In view of \eqref{rp} and the observations stated in Lemma~\ref{Q} we shall not distinguish between \mbox{$Q_\lambda\in\ml(E_0)$} and \mbox{$Q_\lambda\vert_{E_\theta}\in\ml(E_\theta)$} in the sequel  if $\theta\in [0,1)$.\\

The arguments used in the proof of \cite[Lem.2.2]{WalkerAMPA11} reveal:

\begin{lem}\label{L0}
Let $I=\R$ if $a_m<\infty$ and $I=(-\varpi,\infty)$ if $a_m=\infty$.
Then the mapping  $$[\lambda\mapsto r(Q_\lambda)]:I\rightarrow (0,\infty)$$ is continuous, strictly decreasing, and \mbox{$\lim_{\lambda\rightarrow\infty}r(Q_\lambda) =0$}. If $a_m<\infty$, then \mbox{$\lim_{\lambda\rightarrow-\infty}r(Q_\lambda) =\infty$}.
\end{lem}

Next, we characterize the resolvent of $-\A$.

\begin{lem}\label{L1}
Consider $\lambda\in\C$ such that $\mathrm{Re}\, \lambda>-\varpi+\zeta$ and suppose that $1-Q_\lambda\in \ml(E_0)$ is boundedly invertible. Then
\bqn\label{inv}
\big[(\lambda+\A)^{-1}\phi\big](a)=\int_0^a \Pi_\lambda(a,\sigma)\,\phi(\sigma)\ \rd \sigma +\Pi_\lambda(a,0)(1-Q_\lambda)^{-1}\int_0^{a_m} b(s)\int_0^s \Pi_\lambda(s,\sigma)\,\phi(\sigma)\ \rd \sigma\,\rd s
\eqn
for $a\in J$ and $\phi\in\E_0$.
\end{lem}

\begin{proof} By \eqref{23}, any $\lambda\in\C$ with $\mathrm{Re}\, \lambda>-\varpi+\zeta$ belongs to the resolvent set of $-\A$, so it follows from the Laplace transform formula and \eqref{100} that for $\phi\in\E_0$ and a.a. $a\in J$ we have
\bqnn
\big[(\lambda+\A)^{-1}\phi\big](a)=\int_0^\infty e^{-\lambda t}\, \big[S(t)\, \phi\big](a)\ \rd t=\int_0^a \Pi_\lambda(a,t)\,\phi(t)\ \rd t +\Pi_\lambda(a,0) \int_0^\infty e^{-\lambda t} B_\phi(t)\ \rd t
\ .
\eqnn
Next, from \eqref{66},
$$
\Psi:=\int_0^\infty e^{-\lambda t} B_\phi(t)\ \rd t\in E_0
$$
and, on using \eqref{100} and \eqref{6a}, we obtain
\bqnn
\begin{split}
\Psi&=\int_0^{a_m}b(a)\int_0^\infty e^{-\lambda t}\, \big[S(t)\, \phi\big](a)\ \rd t\,\rd a\\
&=\int_0^{a_m} b(a)\, \Pi_\lambda(a,0)\ \rd a\, \Psi +\int_0^{a_m} b(a)\int_0^a  \Pi_\lambda(a,t)\, \phi(t) \,\rd t\,\rd a\ ,
\end{split}
\eqnn
that is,
$$
 \Psi= (1-Q_\lambda)^{-1}\int_0^{a_m} b(a)\int_0^a  \Pi_\lambda(a,t)\, \phi(t) \,\rd t\,\rd a
$$
from which the claim follows.
\end{proof}

Observe that Lemma~\ref{L1} also holds without assumption \eqref{12} on maximal regularity of $-A$ and for $\phi\in L_1(J,E_0)$, i.e. for $p=1$. However, \eqref{12} allows us to interpret formula~\eqref{inv} in the correct functional setting:

\begin{rem}\label{R1}
Let $\phi\in\E_0$, let  $\lambda\in\C$ be such that $\mathrm{Re}\, \lambda>-\varpi+\zeta$, and suppose that $1-Q_\lambda\in \ml(E_0)$ is boundedly invertible. Note that $(1-Q_\lambda)^{-1}\in \ml(E_\varsigma)$ by Lemma~\ref{Q}. Then, by \eqref{12},
\bqn\label{490}
(\lambda+\A)^{-1}\phi= v_\lambda\phi+w_\lambda\phi\ ,
\eqn
where maximal regularity of $A_\lambda$ implies that $v_\lambda\phi\in \E_1$, given by
$$
(v_\lambda\phi) (a):= \int_0^a \Pi_\lambda(a,\sigma)\,\phi(\sigma)\ \rd \sigma\ , \quad a\in J\ ,
$$
is the unique solution  to the Cauchy problem
$$
\partial_a v+A_\lambda v=\phi\ ,\quad a\in (0,a_m)\ ,\qquad v(0)=0\ ,
$$
and $w_\lambda\phi\in \E_1$, given by
$$
(w_\lambda\phi) (a):=\Pi_\lambda(a,0)(1-Q_\lambda)^{-1}\int_0^{a_m} b(s)\, (v_\lambda\phi)(s)\,\rd s\ ,\quad a\in J\ ,
$$
is the unique solution to the Cauchy problem
$$
\partial_a  w +A_\lambda w=0\ ,\quad a\in (0,a_m)\ ,\qquad w(0)= (1-Q_\lambda)^{-1}\int_0^{a_m}b(s)\, (v_\lambda\phi)(s)\,\rd s\in E_\varsigma\ .
$$
\end{rem}

The characterization of the generator $-\A$ of the semigroup $\{S(t)\, ;\, t\ge 0\}$ from Theorem~\ref{Prop1} is now straightforward.

\begin{thm}\label{C1}
$\phi\in \E_0$ belongs to the domain $\mathrm{dom}(-\A)$ of $-\A$ if and only if $\phi\in\E_1$ with 
\bqn\label{AW}
 \phi(0)=\int_0^{a_m}b(a)\,\phi(a)\ \rd a\ .
 \eqn
Moreover, $\A \phi=\partial_a\phi +A\phi$ for $\phi\in
\mathrm{dom}(-\A)$.
\end{thm}

\begin{proof}
By Lemma~\ref{L0}, we can choose $\lambda> -\varpi+\zeta$ such that  $1-Q_\lambda\in \ml(E_0)$ is boundedly invertible. Thus $\lambda$ belongs to the resolvent set of $-\A$ by Theorem~\ref{Prop1}. Remark~\ref{R1} easily gives $\mathrm{dom}(-\A)\subset\E_1$. Moreover, if $\psi\in\E_0$ and $\phi:=(\lambda+\A)^{-1}\psi\in \mathrm{dom}(-\A),$ then $\phi(0)\in E_\varsigma$ by \eqref{6} and $$ \phi(0)= (w_\lambda\psi)(0)= (1-Q_\lambda)^{-1}\int_0^{a_m}b(a)\, (v_\lambda\psi)(a)\ \rd a\ .
$$
The same calculations as in the proof of Lemma~\ref{L1} yield
$$
\int_0^{a_m}b(a)\, \phi(a)\ \rd a=\int_0^\infty e^{-\lambda t}\int_0^{a_m}b(a)\, [S(t)\psi](a)\ \rd a\,\rd t = (1-Q_\lambda)^{-1}\int_0^{a_m}b(a)\, (v_\lambda\psi)(a)\ \rd a =\phi(0)\ .
$$
Conversely, if $\phi \in  \E_1$ satisfies \eqref{AW} then $\psi:=(\partial_a+A_\lambda)\phi\in\E_0$ by \eqref{10a}
and, since $\phi(t)\in E_1$ for a.a. $t\in J$ we have
$$
\frac{\partial}{\partial t}\big(\Pi_\lambda(a,t)\phi(t)\big)=\Pi_\lambda(a,t)\psi(t)
$$
for a.a. $t\in J$ and $a>t$ due to the fact that $\Pi_\lambda$ is the parabolic evolution operator for $A_\lambda$. Integration with respect to $t$ gives
$$
(v_\lambda\psi)(a)=\phi(a)-\Pi_\lambda(a,0)\phi(0)
$$
from which
$$
(w_\lambda\psi)(a)= \Pi_\lambda(a,0)(1-Q_\lambda)^{-1}\int_0^{a_m} b(s)\, \big[\phi(s)-\Pi_\lambda(s,0)\phi(0)\big] \,\rd s = \Pi_\lambda(a,0)\,\phi(0)
$$
for $a\in J$, whence $$\phi=v_\lambda\psi+w_\lambda\psi=(\lambda+\A)^{-1}\psi\in \mathrm{dom}(-\A)\ .$$
Finally, $(\lambda+\A)\phi=\psi=(\partial_a+A_\lambda)\phi$ and the proof is complete.
\end{proof}

\begin{rem}\label{RR}
Theorem~\ref{Prop1} and Theorem~\ref{C1} show that for any initial value $\phi\in\E_1$ satisfying \eqref{AW}, the unique solution $u\in C(\R^+,\E_1)\cap C^1(\R^+,\E_0)$ to \eqref{1}-\eqref{3} is given by $u(t)=S(t)\phi$, $t\ge 0$, with $S(t)$ defined in \eqref{100}. If $\phi$ is only in $\E_0$, then $u(t)=S(t)\phi$, $t\ge 0$, defines a mild solution in $C(\R^+,\E_0)$. Moreover, $u(t)\in\E_0^+$ for $t\ge 0$ if $\phi\in\E_0^+$.
\end{rem}

\section{Stability of the Trivial Solution and Asynchronous Exponential Growth}

We now shall characterize the growth bound $\omega(-\A)$ of $-\A$. We first characterize the point spectrum of $-\A$ and extend formula~\eqref{inv} to a larger class of $\lambda$ values. 

\begin{lem}\label{L2}
(i) Let $\lambda\in\C$ with $\mathrm{Re}\, \lambda>-\varpi$ if $a_m=\infty$ and let $m\in \N\setminus\{0\}$. Then $\lambda\in \sigma_p(-\A)$ with geometric multiplicity $m$ if and only if $1\in \sigma_p(Q_{\lambda})$ with geometric multiplicity $m$.\\
(ii) Formula \eqref{inv} holds for any $\lambda\in \varrho(-\A)$ provided $\mathrm{Re}\, \lambda>-\varpi$ if $a_m=\infty$. 
\end{lem}

\begin{proof}
(i) Let $\lambda\in\C$ with $\mathrm{Re}\, \lambda>-\varpi$ if $a_m=\infty$. Suppose $\lambda\in \sigma_p(-\A)$ has geometric multiplicity $m$ so that there are linearly independent $\phi_1,...,\phi_m\in\mathrm{dom}(-\A)$ with $(\lambda+\A)\phi_j=0$ for $j=1,...,m$. From Theorem~\ref{C1} we deduce 
$$\phi_j(a)=\Pi_\lambda(a,0)\phi_j(0)\qquad\text{with}\qquad \phi_j(0)=Q_\lambda \phi_j(0)\ .$$
Hence, $\phi_1(0),...\phi_m(0)$ are necessarily linearly independent eigenvectors of $Q_\lambda$ corresponding to the eigenvalue~$1$. 
Now, suppose $1\in \sigma_p(Q_{\lambda})$ has geometric multiplicity $m$ so that there are linearly independent $\Phi_1,...\Phi_m\in E_\varsigma$ with $Q_\lambda\Phi_j=\Phi_j$ for $j=1,...,m$. Put $\phi_j:=\Pi_\lambda(\cdot,0)\Phi_j\in\E_1$ and note that, for $j=1,...,m$, 
$$
\partial_a\phi_j+A_\lambda\phi_j=0\ ,\quad
\int_0^{a_m}b(a)\, \phi_j(a)\ \rd a=Q_\lambda \Phi_j=\Phi_j=\phi_j(0)\ .
$$
Thus $\phi_j\in\mathrm{dom}(-\A)$ and $(\lambda+\A)\phi_j=0$ by Theorem~\ref{C1}, i.e. $\lambda\in \sigma_p(-\A)$. If $\alpha_1,...,\alpha_m$ are any scalars, the unique solvability of the Cauchy problem
$$
\partial_a\phi+A_\lambda\phi=0\ ,\quad a\in (0,a_m)\ ,\qquad \phi(0)=\sum_j \alpha_j\, \Phi_j
$$
ensures that $\phi_1,...,\phi_m$ are linearly independent. This proves (i).\\
(ii) Let $\lambda\in \varrho(-\A)$ with $\mathrm{Re}\, \lambda>-\varpi$ if $a_m=\infty$. Then $A_\lambda$ has maximal regularity due to \eqref{12} and, by Theorem~\ref{C1}, $(\lambda+\A)\psi=\phi$ with $\phi\in\E_0$ and $\psi\in\E_1$ if and only if $(\partial_a+A_\lambda)\psi=\phi$ with $$\psi(0)=\int_0^{a_m}b(a)\,\psi(a)\ \rd a\ ,$$ 
that is,
\begin{align*}
\psi(a)&=\Pi_\lambda(a,0)\psi(0)+\int_0^a\Pi_\lambda(a,\sigma)\,\phi(\sigma)\ \rd \sigma\ ,\\
(1-Q_\lambda)\psi(0)&=\int_0^{a_m}b(a)\int_0^a\Pi_\lambda(a,\sigma)\,\phi(\sigma)\ \rd\sigma\,\rd a\ .
\end{align*}
Since $\lambda\in \varrho(-\A)$ and since $Q_\lambda$ is compact, (i) ensures that $1\in\varrho(Q_\lambda)$, hence $1-Q_\lambda$ is invertible and so
$$
\psi(0)=(1-Q_\lambda)^{-1}\int_0^{a_m}b(a)\int_0^a\Pi_\lambda(a,\sigma)\,\phi(\sigma)\ \rd\sigma\,\rd a\ .
$$
As  $\psi=(\lambda+\A)^{-1}\phi$, this gives formula \eqref{inv}.
\end{proof}

Recall that the $\alpha$-growth bound $\omega_1(-\A)$ of $-\A$ is defined by
$$
\omega_1(-\A):=\lim_{t\rightarrow\infty}\frac{1}{t} \log \big(\alpha(S(t))\big)\ ,
$$
where $\alpha$ denotes Kuratowski's measure of non-compactness. That is, if $B$ is a subset of a normed vector space $X$, then $\alpha(B)$ is defined as the infimum over all $\delta>0$ such that $B$ can be covered with finitely many sets of diameter less than $\delta$, and if $T$ is a bounded operator on $X$, then $\alpha(T)$ is the infimum over all $\varepsilon>0$ such that $\alpha(T(B))\le \varepsilon \alpha(B)$ for any bounded set $B\subset X$. Recall that $\omega_1(-\A)\le \omega(-\A)$.\\

We next provide bounds on $\omega_1(-\A)$.

\begin{lem}\label{p4} 
There holds
$$ 
\sup\{\mathrm{Re}\, \lambda \, ;\, \lambda\in\sigma_e(-\A)\}\, \le \, \omega_1(-\A)\, \le \, -\varpi\ .
$$
Moreover, if $a_m<\infty$, then $\omega_1(-\A)=-\infty$ and the semigroup $\{S(t)\,;\, t\ge 0\}$ is eventually compact.
\end{lem}

\begin{proof}
The first inequality of the assertion is generally true for strongly continuous semigroups \cite[Prop.4.13]{WebbSpringer}. We thus merely have to show that $\omega_1(-\A) \le  -\varpi$ what can be done along the lines of the scalar case \cite[Thm.4.6]{WebbSpringer}: Let $t>0$ and write $S(t)=U(t)+W(t)$,  where $U(t), W(t)\in\ml(\E_0)$ are defined as
\bqnn
     \big[U(t)\, \phi\big](a)\, :=\, \left\{ \begin{aligned}
    &0\ ,& & a\in (0,t)\ ,\\
    & \big[S(t)\, \phi\big](a)\ , & & a\in (t,a_m)\ ,
    \end{aligned}
   \right. \,\qquad
    \big[W(t)\, \phi\big](a)\, :=\, \left\{ \begin{aligned}
    &\big[S(t)\, \phi\big](a)\ ,& & a\in (0,t)\ ,\\
    & \ 0\ , & & a\in (t,a_m)\ ,
    \end{aligned}
   \right.
    \eqnn
for $a\in J$, $\phi\in\E_0$. Observing that $\alpha(S(t))\le \alpha(U(t))+\alpha(W(t))$ and, by \eqref{10aa} and \eqref{100},
\bqn\label{111}
\alpha(U(t))\le \|U(t)\|_{\ml(\E_0)}\le Me^{-\varpi t}\ ,\quad t<a_m\ ,
\eqn
the assertion follows from the definition of $ \omega_1(-\A)$ provided we can show  that $\alpha(W(t))=0$. For this it suffices to show that if $B$ is any bounded subset of $\E_0$, then $W(t)B$ is relatively compact in $\E_0$. We use Kolmogorov's compactness criterion \cite[Thm.A.1]{GutmanSIMA}. Without loss of generality we may assume that $a_m=\infty$. Clearly, Theorem~\ref{Prop1} ensures that $W(t)B$ is bounded in $\E_0$. If $\phi\in\E_0$ and $h>0$, then 
\bqnn
\begin{split}
\int_0^\infty \| [W(t)\phi](a+h)&-[W(t)\phi](a)\|_{E_0}^p\ \rd a \\
& \le  \int_0^{t-h} \| \Pi(a+h,0)-\Pi(a,0)\|_{\ml(E_\xi,E_0)}^p\, \|B_\phi(t-a-h)\|_{E_\xi}^p\ \rd a\\ 
& \quad + \int_0^{t-h} \|\Pi(a,0)\|_{\ml(E_0)}^p\, \|B_\phi(t-a-h)-B_\phi(t-a)\|_{E_0}^p\ \rd a\\ 
& \quad  + \int_{t-h}^t \| \Pi(a,0)\|_{\ml(E_0)}^p\, \|B_\phi(t-a)\|_{E_0}^p\ \rd a\ ,
\end{split}
\eqnn
where $\xi\in (0,1/p)$.
On using \eqref{10aa} and Lemma~\ref{A} it is readily seen that the second and third integral on the right side tend to zero as $h\rightarrow 0$, uniformly with respect to $\phi\in B$. For the first integral we use the fact (see \cite[II.Eq.(5.3.8)]{LQPP}) that
$$
\| \Pi(a+h,0)-\Pi(a,0)\|_{\ml(E_\xi,E_0)}\le c(t)h^{\xi}\ ,\quad a\le t\ ,
$$
to obtain from Lemma~\ref{A} the estimate
$$
\int_0^{t-h} \| \Pi(a+h,0)-\Pi(a,0)\|_{\ml(E_\xi,E_0)}^p\, \|B_\phi(t-a-h)\|_{E_\xi}^p\ \rd a\le c(t)^p h^{\xi p}\int_0^{t-h} (t-a-h)^{-\xi p}\ \rd a\ \|\phi\|_{\E_0}^p
$$
with right hand side tending to zero as $h\rightarrow 0$, uniformly with respect to $\phi\in B$. Next, by \eqref{10aa}, \eqref{100}, and Lemma~\ref{A},
$$
\|[W(t)\phi](a)\|_{E_\xi}\le c(B,t)\, (t-a)^{-\xi}\ ,\quad a<t\ ,
$$
for some constant $c(B,t)$. Given $\varepsilon\in (0,t)$ let $R_\varepsilon$ be the $E_0$-closure of the ball in $E_\xi$ centered at $0$ of radius $c(B,t)\varepsilon^{-\xi}$. Then $R_\varepsilon$ is compact in $E_0$ since $E_\xi$ embeds compactly in $E_0$ and
$$
[W(t)\phi](a)\in R_\varepsilon\ ,\quad a\in \R^+\setminus[t-\varepsilon ,t]\ ,\quad \phi\in B\ .
$$
Therefore, \cite[Thm.A.1]{GutmanSIMA} implies that $W(t)B$ is relatively compact in $\E_0$, hence \mbox{$\omega_1(-\A)\le -\varpi$}. Finally, if $a_m<\infty$ and $t>a_m$, then $U(t)=0$ and so $S(t)=W(t)$. This proves the lemma.
\end{proof}

\begin{lem}\label{C2}
Let $\lambda\in\sigma(-\A)$ with $\mathrm{Re}\, \lambda>-\varpi$ if $a_m=\infty$. Then  
$\lambda\in\sigma_p(-\A)\setminus\sigma_e(-\A)$ and $1\in\sigma_p(Q_{\lambda})$.
Moreover, $\lambda$ is isolated in $\sigma(-\A)$  and a pole of the resolvent $[\tau\mapsto (\tau+\A)^{-1}]$. The residue of the resolvent at~$\lambda$,
$$
P_{\lambda}:=\frac{1}{2\pi i}\int_\Gamma (\tau+\A)^{-1}\ \rd \tau\ ,
$$
is a projection on $\E_0$ and $\E_0=\mathrm{im}(P_{\lambda})\oplus \mathrm{im}(1-P_{\lambda})$ with $\mathrm{im}(P_{\lambda})=\mathrm{ker}(\lambda+\A)^m$, where $\Gamma$ is a positively oriented closed curve in the complex plane such that no point in $\sigma(-\A)$ lies in or on $\Gamma$ and  $m\in\N$ is the order of the pole $\lambda$.
\end{lem}

\begin{proof} 
Let $\lambda\in\sigma(-\A)$ with $\mathrm{Re}\, \lambda>-\varpi$ if $a_m=\infty$.  
Since $-\A$ is closed and  $\lambda\in\sigma(-\A)\setminus\sigma_e(-\A)$ by Lemma~\ref{p4}, it follows from \cite{Browder} that $\lambda\in\sigma_p(-\A)$ is isolated in $\sigma(-\A)$  and a pole of the resolvent $[\tau\mapsto (\tau+\A)^{-1}]$. In particular, $\lambda\in\sigma_p(-\A)$ implies $1\in\sigma_p(Q_{\lambda})$ by Lemma~\ref{L2}(i). Since $\lambda$ is isolated in $\sigma(-\A)$, the remaining assertions follow from Laurent series theory described in \cite{Kato,Yosida}, for details we refer to \cite[Prop.4.8, Prop.4.11]{WebbBook}. 
\end{proof}

The following characterization of the peripheral spectrum $\sigma_0(-\A)$ in terms of the spectral bound $s(-\A)$ in the Banach lattice $\E_0$ is a useful tool for our purpose. Its proof is based on \cite{Greiner84} and found in \cite[Prop.2.5]{Webb_TAMS_87}:

\begin{prop}\label{LLL}
If $s(-\A)>\omega_1(-\A)$, then $\sigma_0(-\A)=\{s(-\A)\}$.
\end{prop}

We now give a criterion for the spectral bound to be negative. Note that this result implies the global asymptotic stability of the trivial solution to \eqref{1}-\eqref{2}.

\begin{thm}\label{T1}
If  $r(Q_0)<1$, then $\omega(-\A)<0$.
\end{thm}

\begin{proof}
Let $\hat{\lambda}_0:=s(-\A)$. Since $r(Q_0)<1$ there is $\delta>0$ such that $r(Q_{-\delta})<1$ by Lemma~\ref{L0}. Fix $\varepsilon\in (0, \delta)$ with $\varepsilon<\varpi$ if $a_m=\infty$, see \eqref{10}. Suppose there is $\lambda\in\C$ with $\mathrm{Re}\, \lambda \ge -\varepsilon$ and $1\in \sigma_p(Q_\lambda)$. Then $\lambda\in\sigma(-\A)$ due to Lemma~\ref{L2}, whence $\hat{\lambda}_0\ge -\varepsilon$. Since $\varepsilon<\varpi$ if $a_m=\infty$, Lemma~\ref{p4} and Proposition~\ref{LLL} entail \mbox{$\sigma_0(-\A)=\{\hat{\lambda}_0\}$}. Invoking  Lemma~\ref{L2} again we see that $1\in\sigma_p(Q_{\hat{\lambda}_0})$ and so $r(Q_{\hat{\lambda}_0})\ge 1$.  Lemma~\ref{L0} then gives  $\hat{\lambda}_0<-\delta$ contradicting $\hat{\lambda}_0\ge -\varepsilon$ and $\varepsilon <\delta$. Consequently,
\bqn\label{u}
\sup\big\{\mathrm{Re}\,\lambda\,;\, 1\in \sigma_p(Q_\lambda)\big\} \le -\varepsilon <0\ .
\eqn
Recall e.g. from \cite[Prop.4.13]{WebbBook} that
$$
\omega(-\A)=\max\big\{\omega_1(-\A), \sup_{ \lambda\in\sigma(-\A)\setminus \sigma_e(-\A)}   \mathrm{Re}\, \lambda\big\}\ .
$$
Since $\varpi>0$ if $a_m=\infty$, the assertion is an immediate consequence of \eqref{u}, Lemma~\ref{p4}, and Lemma~\ref{C2}.
\end{proof}

Next, we provide a criterion for asynchronous exponential growth of the semigroup $\{S(t);t\ge 0\}$ which is analogous to the scalar case $A\equiv \mu$ in \cite{WebbBook}. For similar results in the spatially inhomogeneous setting we refer to \cite{ThiemeDCDS,RhandiSchnaubelt_DCDS99}. We first need an auxiliary result.

\begin{lem}\label{L12}
Let $\lambda_0\in\R$ with $\lambda_0>-\varpi$ if $a_m=\infty$ be such that $r(Q_{\lambda_0})=1$. Then $\lambda_0$ is a simple eigenvalue of~$-\A$.
\end{lem}

\begin{proof}
Referring to Lemma~\ref{Q} there is a quasi-interior eigenvector $\Phi_0\in E_1$ of $Q_{\lambda_0}$ corresponding to the simple eigenvalue  $r(Q_{\lambda_0})=1$. By Theorem~\ref{C1} and Lemma~\ref{L2}, $\mathrm{ker}(\lambda_0+\A)$ is one-dimensional and spanned by \mbox{$\varphi:=\Pi_{\lambda_0}(\cdot,0)\Phi_0$}. It thus remains to show that
$\mathrm{ker}(\lambda_0+\A)^2\subset \mathrm{ker}(\lambda_0+\A)$. Let \mbox{$\psi\in\mathrm{ker}(\lambda_0+\A)^2$} and set $$\phi:=(\lambda_0+\A)\psi\in \mathrm{ker}(\lambda_0+\A)\ .$$ Then $\phi=\xi\varphi$ for some $\xi\in\R$. Suppose $\xi\not= 0$, so without loss of generality $\xi>0$. Let $\tau>0$ be such that $\tau\Phi_0+\psi(0)\in E_\varsigma^+\setminus\{0\}$ and put $q:=\tau\varphi+\psi\in\mathrm{dom}(-\A)$. Then $(\lambda_0+\A)q=\phi$ and, from Theorem~\ref{C1},
$$
q(a)=\Pi_{\lambda_0}(a,0) q(0)+\xi\int_0^a\Pi_{\lambda_0}(a,\sigma)\,\Pi_{\lambda_0}(\sigma,0)\,\Phi_0\ \rd \sigma= \Pi_{\lambda_0}(a,0) q(0)+ a\,\xi\, \Pi_{\lambda_0}(a,0)\,\Phi_0
$$
and
$$
q(0)=\int_0^{a_m} b(a)\, q(a)\ \rd a \ .
$$
Plugging the former into the second formula yields
$$
(1-Q_{\lambda_0}) q(0)=\xi \int_0^{a_m} b(a)\, a\, \Pi_{\lambda_0}(a,0)\,\Phi_0\ \rd a\ .
$$
As $q(0)$ and the right hand side are both positive and nonzero, we derive from \cite[Cor.12.4]{DanersKochMedina} a contradiction to \mbox{$r(Q_{\lambda_0})=1$}. Consequently, $\xi=0$ and the claim follows because now $\phi=0$.
\end{proof}

Exponential asynchronous growth of the semigroup  $\{S(t); t\ge 0\}$ given in \eqref{100} is now an easy consequence of~\cite{Webb_TAMS_87}. Recall from Lemma~\ref{C2} that $P_{\lambda_0}:\E_0\rightarrow \mathrm{ker}(\lambda_0+\A)$ is a projection with rank one.\

\begin{thm}\label{C14}
Suppose that $r(Q_\alpha)>1$ for some $\alpha\in\R$ with  $\alpha>-\varpi$ if $a_m=\infty$.
Then $\{S(t)\,;\,  t\ge 0 \}$ has asynchronous exponential growth with intrinsic growth constant $\lambda_0$, that is,
$$
e^{-\lambda_0 t}\, S(t)\longrightarrow P_{\lambda_0}\quad\text{in}\quad \ml(\E_0)\quad \text{as}\quad  t\rightarrow\infty\ ,
$$
where $\lambda_0>\alpha$ is the unique number satisfying $r(Q_{\lambda_0})=1$.
\end{thm}

\begin{proof}
By Lemma~\ref{L0} there is a unique $\lambda_0>\alpha$ such that $r(Q_{\lambda_0})=1$. Let $\hat{\lambda}_0:=s(-\A)$ denote the spectral bound of $-\A$. Then, owing to Lemma~\ref{Q} and Lemma~\ref{C2}, \mbox{$\hat{\lambda}_0\ge \lambda_0$} and so \mbox{$\hat{\lambda}_0>\omega_1(-\A)$} according to Lemma~\ref{p4} since $\lambda_0>-\varpi$ if $a_m=\infty$ or $\omega_1(-\A)=-\infty$ if $a_m<\infty$. Thus, \mbox{$\sigma_0(-\A)=\{\hat{\lambda}_0\}$} by Proposition~\ref{LLL}, and then $1\in\sigma_p(Q_{\hat{\lambda}_0})$ by Lemma~\ref{C2} from which $1\le r(Q_{\hat{\lambda}_0})$. However, due to $\hat{\lambda}_0\ge \lambda_0$ and Lemma~\ref{L0} we have \mbox{$r(Q_{\hat{\lambda}_0})\le r(Q_{\lambda_0})=1$}, whence $\hat{\lambda}_0=\lambda_0$. Consequently,  $\sigma_0(-\A)=\{\lambda_0\}$ and $\lambda_0>\omega_1(-\A)$. Finally, Lemma~\ref{L12} together with Lemma~\ref{C2} imply that $\lambda_0$ is a simple pole of the resolvent $(\tau+\A)^{-1}$. The assertion then follows from \cite[Prop.2.3]{Webb_TAMS_87}.
\end{proof}

To derive a formula for the projection $P_{\lambda_0}:\E_0\rightarrow \mathrm{ker}(\lambda_0+\A)$ from  Theorem~\ref{C14} observe that there is a quasi-interior element $\Phi_0$ in $E_1$ such that $\mathrm{ker}(1-Q_{\lambda_0})=\mathrm{span}\{\Phi_0\}$ and $\mathrm{ker}(\lambda_0+\A)=\mathrm{span}\{\Pi_{\lambda_0}(\cdot,0)\Phi_0\}$.
Let $\phi\in\E_0$ be fixed and let $c(\phi)\in\R$ be such that $P_{\lambda_0}\phi=c(\phi)\Pi_{\lambda_0}(\cdot,0)\Phi_0$. Recall that $\lambda_0$ is a simple pole of the resolvent $ (\tau+\A)^{-1}$. Since $v_\lambda$ is holomorphic in $\lambda$, it follows from \eqref{490} and the Residue Theorem that 
$$
P_{\lambda_0}\phi=\lim_{\lambda\rightarrow\lambda_0}(\lambda-\lambda_0)\Pi_\lambda(\cdot,0) (1-Q_\lambda)^{-1} H_\lambda \phi\ ,\qquad H_\lambda\phi:=\int_0^{a_m}b(s)\, \int_0^s\Pi_\lambda(s,\sigma)\,\phi(\sigma)\ \rd\sigma\,\rd s\ .
$$
Let $w^*\in E_0^*$ be a positive eigenfunctional of the dual operator $Q_{\lambda_0}^*$ of $Q_{\lambda_0}$ corresponding to the eigenvalue $r(Q_{\lambda_0})=1$. Then, for $f^*\in\E_0^*$ defined by
$$
\langle f^*,\psi\rangle:=\langle w^*,\int_0^{a_m} b(a)\psi(a)\rd a\rangle\ ,\quad \psi\in\E_0\ ,
$$
we have, due to $Q_{\lambda_0}^*w^*=w^*$, that
\bqnn
\begin{split}
c(\phi)\langle w^*,\Phi_0\rangle &=\langle f^*,P_{\lambda_0}\phi\rangle = \lim_{\lambda\rightarrow\lambda_0}\langle f^*, (\lambda-\lambda_0)\Pi_\lambda(\cdot,0) (1-Q_\lambda)^{-1} H_\lambda \phi\rangle\\
&= \lim_{\lambda\rightarrow\lambda_0}\langle w^*, (\lambda-\lambda_0)(Q_\lambda-1+1) (1-Q_\lambda)^{-1} H_\lambda \phi\rangle\\
&
= \lim_{\lambda\rightarrow\lambda_0}\langle w^*, (\lambda-\lambda_0) (1-Q_\lambda)^{-1} H_\lambda \phi\rangle \ .
\end{split}
\eqnn
Writing
\bqn\label{ii}
H_\lambda\phi=d\big(H_\lambda\phi\big)\Phi_0 \oplus (1-Q_{\lambda_0})g\big(H_\lambda\phi\big)
\eqn
according to the decomposition $E_0=\R\cdot \Phi_0\oplus \mathrm{rg}(1-Q_{\lambda_0})$, it follows
$$
\lim_{\lambda\rightarrow\lambda_0}\langle w^*, (\lambda-\lambda_0) (1-Q_\lambda)^{-1} H_\lambda \phi\rangle=
d\big(H_{\lambda_0}\phi\big) \lim_{\lambda\rightarrow\lambda_0}\langle w^*, (\lambda-\lambda_0) (1-Q_\lambda)^{-1} \Phi_0\rangle\ 
$$
since $Q_\lambda$ is continuous in $\lambda$.
But, from \eqref{ii},
$$
\langle w^*,H_{\lambda_0}\phi\rangle =d\big(H_{\lambda_0}\phi\big)\, \langle w^*, \Phi_0\rangle
$$
since $Q_{\lambda_0}^*w^*=w^*$, whence $d\big(H_{\lambda_0}\phi\big)=\xi \langle w^*,H_{\lambda_0}\phi\rangle$ with $\xi^{-1}= \langle w^*, \Phi_0\rangle$. Similarly, decomposing $$Z_\lambda:=(\lambda-\lambda_0) (1-Q_\lambda)^{-1} \Phi_0$$ we  find
$$
\lim_{\lambda\rightarrow\lambda_0}\langle w^*,Z_\lambda\rangle =\big(\lim_{\lambda\rightarrow\lambda_0} d\big(Z_\lambda\big)\big) \langle w^*,\Phi_0\rangle.
$$
Gathering these observations, we derive
$$
c(\phi)\,\langle w^*,\Phi_0\rangle = C_0\, \langle w^*,H_{\lambda_0}\phi\rangle  \, \langle w^*, \Phi_0\rangle
$$
for some number $C_0$. Consequently,
$$
P_{\lambda_0}\phi= C_0\,\langle w^*,H_{\lambda_0}\phi\rangle\, \Pi_{\lambda_0}(\cdot,0)\Phi_0\ .
$$
Since $P_{\lambda_0}$ is a projection, i.e. $P_{\lambda_0}^2=P_{\lambda_0}$, the constant $C_0$ is easily computed and we obtain:

\begin{prop}\label{PPP}
Under the assumptions of Theorem~\ref{C14}, the projection $P_{\lambda_0}$ is given by
\bqn\label{PP}
P_{\lambda_0}\phi=\frac{\langle w^*,H_{\lambda_0}\phi\rangle}{\langle w^*,\int_0^{a_m} a b(a)\Pi_{\lambda_0}(a,0)\rd a\,\Phi_0\rangle}\Pi_{\lambda_0}(\cdot,0)\Phi_0\ 
\eqn
for $\phi\in\E_0$, where 
$$
 H_{\lambda_0}\phi=\int_0^{a_m}b(s)\, \int_0^s\Pi_{\lambda_0}(s,\sigma)\,\phi(\sigma)\ \rd\sigma\,\rd s
 $$
and $w^*\in E_0^*$ is a positive eigenfunctional to the dual operator $Q_{\lambda_0}^*$ of $Q_{\lambda_0}$ corresponding to the eigenvalue~$r(Q_{\lambda_0})=1$.
\end{prop}

Situations in which assumptions \eqref{10a}-\eqref{14} are satisfied occur, for instance, when $A$ is a second order elliptic operator in divergence form. We refer to  \cite[Sect.3]{WalkerSIMA09}, \cite[Sect.3]{WalkerAMPA11} for concrete examples and details. Therein one also finds simple settings in which the spectral radii $r(Q_\lambda)$ can be computed easily what allows one to apply Theorem~\ref{T1} and Theorem~\ref{C14}.

\end{document}